\documentclass{article}
\usepackage[utf8]{inputenc}
\usepackage{dsfont}
\usepackage{graphicx}
\usepackage{amsmath}
\usepackage{cancel}
\usepackage{ stmaryrd }
\usepackage{geometry}
\usepackage{amsfonts}
\usepackage{amsthm}
\usepackage{hyperref}
\usepackage[french,english]{babel}
\usepackage{amsfonts,amsmath,latexsym,amsthm,amssymb,euscript,eufrak,graphicx,units,mathrsfs,setspace,stmaryrd,dsfont}

\usepackage{systeme}
\usepackage{float}
\newfloat{figure}{H}{lof}
\floatname{figure}{\figurename}

\newtheorem{prpstn}{Proposition}[section]

\newtheorem{thrm}[prpstn]{Theorem}
\newtheorem{lmm}[prpstn]{Lemma}

\newtheorem{rmrk}[prpstn]{Remark}

\newtheorem{dfntn}[prpstn]{Definition}

\newcommand{\etal}{\textit{et al.}}
\newcommand{\xdif}{\mathrm{d}}
\author{
	 Lorick Huang \footnote{INSA de Toulouse, IMT UMR CNRS 5219, Universit\'e de Toulouse, 135 avenue de Rangueil 31077 Toulouse Cedex 4 France. \; Email: \texttt{lorick.huang@insa-toulouse.fr.}} \and Mahmoud Khabou\footnote{INSA de Toulouse, IMT UMR CNRS 5219, Universit\'e de Toulouse, 135 avenue de Rangueil 31077 Toulouse Cedex 4 France. \; Email: \texttt{mahmoud.khabou@insa-toulouse.fr}} 
}

\title{Convergence of the Discrete-Time Compound Hawkes Process with Exponential or Erlang Kernel}
\begin{document}
\maketitle

\allowdisplaybreaks
%
%
\begin{abstract} 
	\noindent
	Due to its clustering and self-exciting properties, the Hawkes process has been used extensively in numerous fields ranging from sismology to finance. Since data is often aquired on regular time intervals, we propose a piece-wise constant model based on a Discrete-Time Hawkes Process \textit{(DTHP)}. We prove that this discrete-time model converges to the usual continuous-time Hawkes process as the time-step tends to zero. 
\selectlanguage{french} 
\begin{center} 
	\textbf{Résumé} 
\end{center} 
Les propri\'et\'es d’auto-excitation des processus de Hawkes permettent une alternative de mod\'elisation efficace au processus de Poisson \`a intensit\'e d\'eterministe dans plusieurs domaines d'application comme la finance ou la sismologie. Dans certaines applications, l'acc\`es aux donn\'ees se fait \`a des dates d\'eterministes et non de fa\c con continue dans le temps. Ainsi, seulement une approximation \`a temps discret du processus de Hawkes sur une grille d\'eterministe est observable. Dans cet article nous \'etudions la convergence de cette approximation \`a temps discret lorsque le pas de la subdivision tend vers z\'ero.
\end{abstract}
\noindent
\textbf {Subject Class}: 60J05, 60J25, 60G55.

\section{Introduction and Main Result}
The linear Hawkes process was first introduced in 1971 by Hawkes \cite{MR0358977} as a point process whose intensity exhibits an interesting self-excitation property. Even though Hawkes process has initially contributed to seismology by describing the aftershocks in case of an earthquake, its self-exciting and clustering properties made it a popular model in  financial and actuarial applications.

For instance Errais \etal{} used it to model the cumulative loss due to default in a portfolio of firms \cite{MR2719785}, while Bacry \etal{} used it for measuring the endogeneity of stock markets \cite{bacry2015hawkes}.\\

In the continuous time setting, the Hawkes process is defined as follows.
Consider a probability space $(\Omega, \mathcal F, \mathbb P)$ with a filtration $(\mathcal F_t)_{t\in [0,+\infty)}$ and a sequence of increasing stopping times $0<\theta_1 < \theta_2 <\cdots$.\\
A point process is defined as the counting measure 
$$H_t=H([0,t]):=\sum_{i=1}^{+\infty}\mathds 1_{\theta_i \leq t}.$$
We assume that an event at time $\theta_n$ corresponds to a  financial loss $\zeta_n$. The total loss a time $t$ is the compound process
$$L_t:=\sum_{i=1}^{+\infty} \zeta_i \mathds 1_{\theta_i \leq t} = \sum_{i=1}^{H_t} \zeta_i ,$$
where $\zeta_n$ are independent identically distributed \textit{(i.i.d)} non-negative random variables with an integrable distribution $\nu$ and independent from $(H_t)_{t\in [0,+\infty)}$.\\
\begin{rmrk}
	If $\zeta_n$ are chosen to be deterministic and equal to $1$ then $L_t=H_t$ for all $t \geq 0$. 
\end{rmrk}
The intensity of a point process is a measure of how much it tends to jump at a certain time $t$ and is defined as 
$$\lambda_t=\lim_{\delta t \rightarrow 0}\frac{\mathbb E[H_{t+\delta t -}-H_t|\mathcal F_t]}{\delta t}.$$
In the case of a Hawkes process, the realization of an event causes an increase in the probability of other events. This translates in the intensity as:
\begin{align*}
\lambda_t&=\mu (t) + \int_{[0,t)}\phi(t-s) \xdif L_s,\\
&=\mu (t) + \sum _{\theta_i<t}\phi(t-\theta_i)\zeta_i,
\end{align*}
where $\mu$ is a deterministic non-negative function playing the role of the baseline intensity and $\phi$ is a non-negative decaying kernel. Indeed, more events ($\theta_i$) mean more terms in the sum, thus a higher intensity which in return triggers more events.
Larger losses have a bigger impact on the intensity as well.\\
The condition to avoid instability (i.e. infinite amount of jumps in a finite interval) is  $\|\phi\|_1 \mathbb E [\zeta]<1$. Curious readers can consult \cite{MR3313750} for nearly unstable Hawkes processes (the kernel's norm approaches the limit of instability).\\

In this paper we study the case where the intensity kernel $\phi$ is either an exponential $(\phi(u)=\alpha e^{-\beta u})$ or an Erlang function $(\phi(u)=\alpha u e^{-\beta u})$. The exponential kernel case has been studied extensively in the literature. This is mainly because in this case, the intensity $(\lambda_t)_{t\in [0,+\infty)}$ is a Markov process. For example, Errais \etal{} \cite{MR2719785} derived formulae for the Laplace transform for the Markov Hawkes process. Indeed, if the baseline intensity is chosen to be $\mu(t)=\lambda_{\infty}+(x-\lambda_{\infty}) e^{-\beta{t}}$, with the initial intensity $x\geq 0$ and the parameter $\lambda_{\infty}>0$
, the intensity takes the form:
\begin{equation}
\lambda_t=\lambda_{\infty}+(x-\lambda_{\infty})e^{-\beta{t}}+\int_{[0,t)}\alpha e^{-\beta(t-s)}\xdif L_s, \label {eq:intensity}
\end{equation}
where $\alpha$ and $\beta$ are two positive real numbers such that $\beta> \alpha \mathbb E [\zeta]$.
In this case, the intensity satisfies the following stochastic differential equation (SDE):
\begin{equation*}
(SDE_{exp})
\begin{cases}
\xdif \lambda_t=\beta(\lambda_\infty-\lambda_t)\xdif t +\alpha \xdif L_t,\\
\lambda_0=x.
\end{cases}
\end{equation*}
\begin{rmrk}
	In many cases, the initial intensity $\lambda_0$ is chosen to be equal to the parameter $\lambda_{\infty}$ which yields a constant baseline intensity $\lambda_t=\lambda_{\infty}+\int_{[0,t)}\alpha e^{-\beta (t-s)}\xdif L_s.$
\end{rmrk}
If the kernel is an Erlang function, then the intensity takes the form 
\begin{equation}
\lambda_t=\lambda_{\infty}+(x-\lambda_{\infty})e^{-\beta{t}}+\int_{[0,t)}\alpha (t-s) e^{-\beta(t-s)}\xdif L_s, \label {eq:intensityE}
\end{equation}
It is possible to "Markovize" the intensity by taking an auxiliary process $\xi_t=\int_{[0,t)} \alpha e^{-\beta(t- s)} \xdif L_s$ into account. Thus, in this case as well, the vector $(\lambda_t,\xi_t)_{[0,+\infty)}$ follows the dynamics given by the SDE:
\begin{equation*}
(SDE_{Erl})
\begin{cases}
\xdif \lambda_t=\beta(\lambda_\infty-\lambda_t)\xdif t +\xi_t \xdif t,\\
\xdif \xi_t=-\beta \xi_t \xdif t +\alpha \xdif L_t,\\
\lambda_0=x,\\
\xi_0=0.
\end{cases}
\end{equation*}

So far the simulation of the Hawkes process has been based on Ogata's thinning \cite{1056305}, on an immigration clustering approach like in the work of M{\o}ller \etal{} \cite{MR2156552} or in the particular Markov case on the sampling of jumping times such as the algorithm proposed by Dassios \etal{} \cite{MR3084573}.\\
These approaches simulate exactly the jump times of the process on a time continuum. However, in reality data is often recorded on discrete time intervals, \textit{e.g.} every minute, every hour or every day. \\
This motivates the study of Discrete-Time Hawkes Processes \textit{(DTHP)} first introduced by Seol \cite{MR3321518}, where limit theorems have been established as time goes to infinity.\\
In this paper we study the behaviour as the size of the time step goes to zero instead.\\ The intensity (in the exponential kernel case) or the intensity-auxiliary process vector (in the Erlang kernel case) of this DTHP is considered as piece-wise constant process constructed from a Markov chain on the time grid (\textit{cf.} figure ~\ref{fig:subdivision}).
\begin{rmrk}
	Knowing the intensity is sufficient for the reconstruction of  $(L_t)_{t\in [0,+\infty)}$. This can be seen on figure \ref{fig:dassios} taken from \cite{MR3084573}. This is why we focus on the intensity from now on. The loss process $(L_t)_{t\in [0,+\infty)}$ is obtained by adding an independent copy of $\zeta$ at every jumping time.
	\begin{figure}[h!]
		\centering
		\includegraphics[width=110mm]{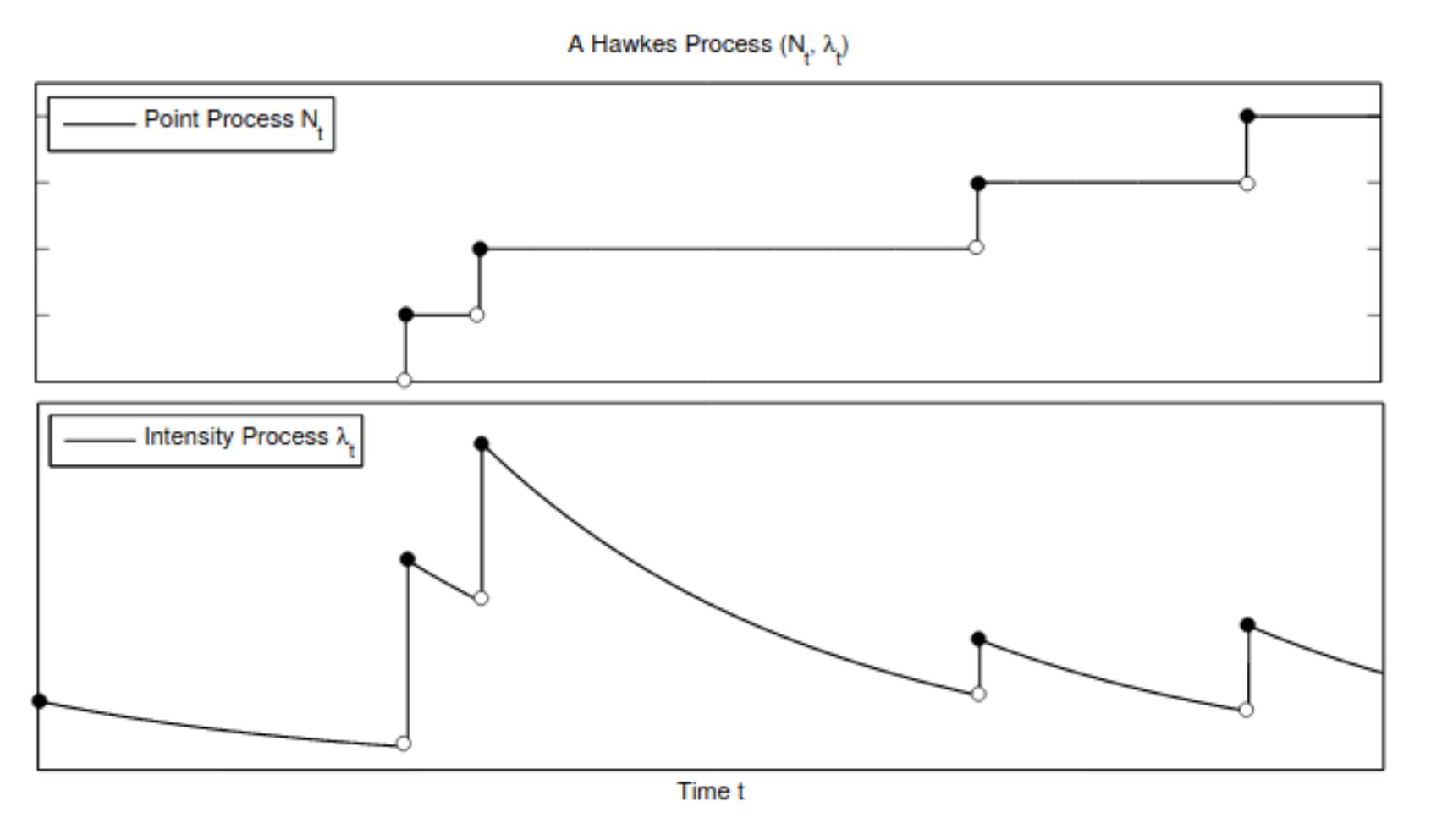}
		\caption{Hawkes process with exponential decaying intensity $(N_t,\lambda_t)$.}
		\label{fig:dassios}
	\end{figure}
\end{rmrk}
The main result is to show that the intensity (\textit{resp.} intensity-auxiliary process vector) converges weakly to the continuous time Hawkes intensity (\textit{resp.} to the intensity-auxiliary process vector) in the Skorokhod topology on $[0,+\infty)$ as the grid becomes finer and finer.\\

Let $\zeta$ be a positive random variable with finite expectation and let $\nu$ be its distribution. Let $\alpha, \beta, \lambda_{\infty}\in \mathbb R ^*_+$ such that $\alpha \mathbb E [\zeta]<\beta$ (exponential kernel) or $\alpha \mathbb E [\zeta]< \beta ^2$ (Erlang kernel) and $x>0$.\\
Let $[0,T], 0<T<+\infty$ be a time interval, $N\in \mathbb N^*$ and $(t_i^N:=\frac{iT}{N})_{i\in \llbracket 0\cdots N\rrbracket}$ be a grid with a step $h_N=\frac{T}{N}$. In some cases we refer to $h_N$ by $h$ to avoid clogging up the notation.
\begin{figure}[h!]
	\centering
	\includegraphics[width=100mm]{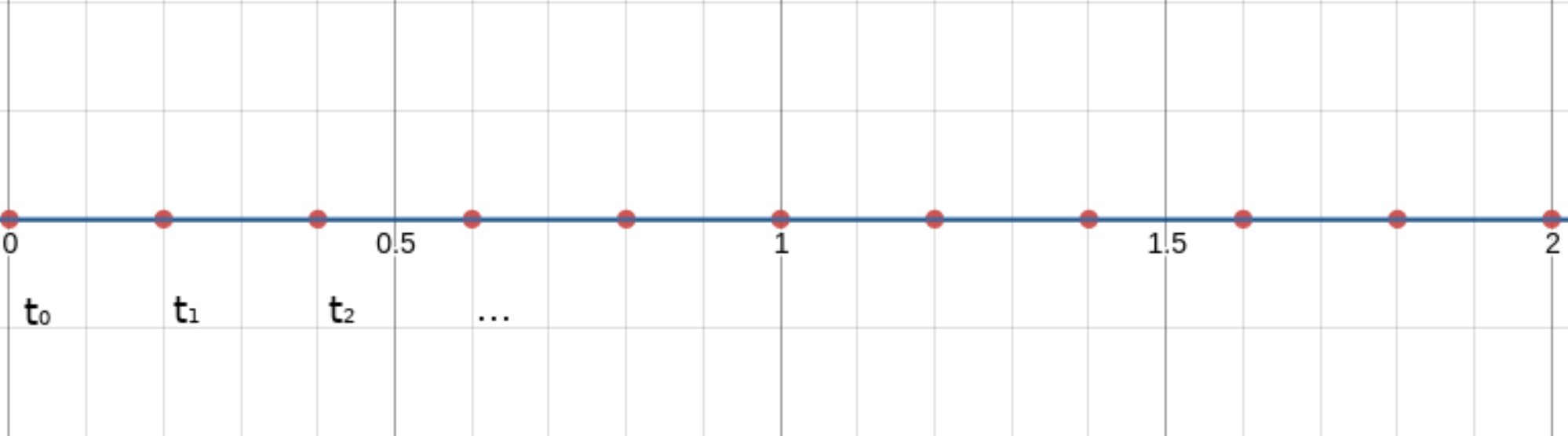}
	\caption{An example of a subdivision with $T=2$ and $N=10$}
	\label{fig:subdivision}
\end{figure}

\begin{dfntn}
	\label{chain}
	Let $(\Omega,\mathcal F,\mathbb P)$ be a probability space.
	Let $N\in \mathbb N^*$ , $T>0$ and a sequence of independent $[0,1]$ uniform random variables $(U^N_k)_{k\in \mathbb N}$ as well as a sequence $(\xi^N_k)_{k\in \mathbb N}$ of iid positive random variables with finite expectation defined on $(\Omega,\mathcal F,\mathbb P)$.
	\begin{enumerate}
		\item \textit{If $\phi$ is an exponential kernel:}
		The Hawkes Markov Chain $(l^N)$ is a Markov chain defined according to the induction rule:
		\begin{equation*}
		(l^N)
		\begin{cases}
		l^N_{k+1}=\lambda_\infty(1-e^{-\beta h})+(l^N_k+\alpha \zeta_{k+1}^N\mathds 1_{U_{k+1}^N<l^N_k\cdot h})e^{-\beta h},\\
		l^N_0=x.
		\end{cases}
		\end{equation*}
		\item \textit {If $\phi$ is an Erlang kernel:} The Hawkes Markov Chain $(l ^N,a ^N)$ is a Markov chain defined according to the induction rule:
		\begin{equation*}
		( l ^N, a ^N)
		\begin{cases}
		l^N_{k+1}=\lambda_\infty(1-e^{-\beta h})+ l^N_k e^{-\beta h}+ a^N_{k+1}h,\\
		a^N_{k+1}= (a^N_{k}+\alpha \zeta_{k+1}^N \mathds{1}_{U^N_{k+1}< l^N_{k}\cdot h})e^{-\beta h},\\
		
		l^N_0=x, \\ a^N_0=0.
		\end{cases}
		\end{equation*}
		
	\end{enumerate}
\end{dfntn}

\begin{dfntn} \label{approx}
	Given $N \in \mathbb N^*$ and $T>0$, the $N$-th DTHP intensity $(\tilde \lambda^N_t)_{t\in [0,+\infty)}$ and the Hawkes auxiliary process $(\tilde \xi_t^N)_{t\in [0,+\infty)}$ (if the kernel is an Erlang function) are defined as the c\`adl\`ag process 
	$$\tilde\lambda^N_t=l^N_{\lfloor \frac{Nt}{T}\rfloor},$$
	$$\tilde\xi^N_t=a^N_{\lfloor \frac{Nt}{T}\rfloor},$$
	where $l^N_{\lfloor \frac{Nt}{T}\rfloor}$ and $a^N_{\lfloor \frac{Nt}{T}\rfloor}$ are defined in \ref{chain}.\\
	This process takes the values of the Markov chain on the grid points. Indeed
	$$\tilde \lambda _{t_i}^N=l^N_{\lfloor \frac{iTN}{TN}\rfloor}=l^N_i \text{ and } \tilde \xi _{t_i}^N=a^N_{\lfloor \frac{iTN}{TN}\rfloor}=a^N_i.$$
\end{dfntn}

The following theorem, which will be proven in the following sections, states the main result:
\begin{thrm}
	\label{MainResult}
	Let $(H_t)_{t\in [0,+\infty)}$ be a Hawkes process, $(L_t)_{t\in [0,+\infty)}$ its loss and $(\lambda_t)_{t\in [0,+\infty)}$ its intensity.\\
	\begin{enumerate}
		\item \textit {If $\phi$ is an exponential kernel:}
		Let $(\tilde \lambda_t^N)_{t\in [0,+\infty)}$ be an $N$-th DTHP intensity (defined in \ref{approx}). Then we have the convergence
		$$(\tilde \lambda_t^N)_{t\in [0,+\infty)}\Longrightarrow _{N\rightarrow +\infty} (\lambda_t)_{t\in [0,+\infty)}$$
		weakly in the Skorokhod space $D_{\mathbb R_+} [0,+\infty)$, the set of all right continuous with left limits (c\`adl\`ag) non-negative functions on ${\mathbb R_+}=[0,+\infty)$.
		
		\item \textit {If $\phi$ is an Erlang kernel:}
		Let $(\tilde \lambda_t^N,\tilde \xi_t^N)_{t\in [0,+\infty)}$ be an  $N$-th DTHP intensity and auxiliary process (defined in \ref{approx}). Then we have the convergence
		$$(\tilde \lambda_t^N,\tilde \xi_t^N)_{t\in [0,+\infty)}\Longrightarrow _{N\rightarrow +\infty} (\lambda_t,\xi_t)_{t\in [0,+\infty)}$$
		weakly in the Skorokhod space $D_{\mathbb R_+^2} [0,+\infty)$.
	\end{enumerate}
\end{thrm}
\begin{rmrk}
	Normally the intensity is a c\`agl\`ad process because it should be predictable (beyond the scope of this paper) but we work with the c\`adl\`ag version because the convergence results that we have in \cite{MR838085} as well as the Markov generator expression in \cite{MR2719785} are for the c\`adl\`ag version.\\
	Therefore we make the change $\lambda_t\leftarrow \lambda_{t+}=\lim _{\delta \shortdownarrow 0} \lambda_{t+\delta}.$
\end{rmrk}
\section{Preliminary Results}
\subsection{General Notations and Lemmas}
We denote by ${\mathbb R_+}=[0,+\infty)$ and we set $E=\mathbb R_+$ or $\mathbb R^2_+$.  $\hat C(E)$ the space of real continuous functions on $E$ vanishing at infinity.\\

$D_{E}[0,+\infty)$ refers to the set of all right continuous with left limits (c\`adl\`ag) functions  $x:[0,+\infty) \rightarrow E$.\\
On the other hand, c\`agl\`ad is used to refer to left continuous functions with right limits.
\begin{lmm}
	$({\mathbb R_+},|.|)$ is locally compact for the topology induced by the absolute value.
\end{lmm}
\begin{proof}
	$(\mathbb R,|.|)$ is locally compact: every point has a compact neighbourhood.\\
	The topology induced on ${\mathbb R_+}$ is simply the set $Top^+=\{{\mathbb R_+}\cap O, O\in Top\}$ with $Top$ being the usual topology on $\mathbb R$. Thus $[0,1)$ is an open set containing $0$ for $({\mathbb R_+},|.|)$, which means that $[0,1]$ is a compact neighbourhood of $0$. Any $x>0$ has a compact neighbourhood $[x-\epsilon,x+\epsilon]$ for $\epsilon$ small enough.
\end{proof}
\begin{lmm}
	$(\hat C ({\mathbb R_+}),\|.\|)$ is a Banach space for $\|f\|=\sup_{x\in {\mathbb R_+}}|f(x)|.$
\end{lmm}
\begin{proof}
	Let $(f_n)_{n\in \mathbb N}$ be a Cauchy sequence in $\hat C ({\mathbb R_+})$. Let $\epsilon >0$, there exists $M$ such that $\forall n,p \geq M$, $\|f_n-f_p\|\leq \epsilon$. Set $x \in {\mathbb R_+}$, $|f_n(x)-f_p(x)|\leq \|f_n-f_p\|\leq \epsilon$ for $n,p \geq M$. Since $\mathbb R$ is complete, $\big(f_n(x)\big)_{n\in \mathbb N}$ converges for every $x\in {\mathbb R_+}.$ We call the point-wise limit $f(x)$. Set $p\geq M$ and $x\in {\mathbb R_+}$,
	\begin{align*}
	|f_p(x)-f(x)|&=|f_p(x)-\lim _{n\rightarrow +\infty} f_n(x)|,\\
	&=\lim _{n\rightarrow +\infty}|f_p(x)-f_n(x)|,\\
	&\leq \lim _{n\rightarrow +\infty} \|f_n-f_p\|,
	\end{align*}
	and since $n\geq M$ (it goes to infinity) we have $|f_p(x)-f(x)|\leq \epsilon$. Because $M$ is independent from $x$ we have the the uniform convergence $\|f_p-f\|\leq \epsilon$.\\
	Let $n$ be such that $\|f_n-f\|\leq \epsilon$ and $K$ such that $|f_n(x)|\leq \epsilon$ if $x>K$ (remember that the functions vanish at infinity). For all $x\in {\mathbb R_+}$, by the triangle inequality $$|f(x)|\leq |f(x)-f_n(x)|+|f_n(x)|\leq \|f-f_n\|+|f_n(x)|.$$
	If $x>K$ then $|f(x)|\leq 2\epsilon$, which means that $f$ vanishes at infinity.\\
	
	To prove the continuity of the limit function, let $a\in {\mathbb R_+}$ and $n$ such that $\|f_n-f\|<\epsilon$. $f_n$ is continuous at $a$ therefore there exists $\eta >0$ such that $|x-a|< \eta \Longrightarrow |f_n(x)-f_n(a)|<\epsilon.$ By the triangle inequality:
	\begin{align*}
	|f(x)-f(a)|&\leq |f_n(x)-f(x)|+|f_n(x)-f_n(a)|+|f(a)-f_n(a)|,\\
	&\leq 2\|f_n-f\|+|f_n(x)-f_n(a)|,
	\end{align*}
	thus $|f(x)-f(a)|\leq 3 \epsilon$ if $|x-a|< \eta$.\\
	In conclusion, $f$ is continuous and vanishes at infinity thus $\hat C({\mathbb R_+})$ is a Banach space.
\end{proof}
From now on the convergence in $\hat C(E)$ refers to the convergence in the uniform norm $\|f\|=\sup_{x\in E}|f(x)|.$

\begin{lmm}\label{density}
	\label{density2}
	The set of twice continuously differentiable functions with compact support $\hat C^2_c({\mathbb R_+})$ is dense in $\hat C ({\mathbb R_+})$ for the norm $\|.\|$. 
\end{lmm}
\begin{proof}
	Let $K \in \mathbb R_+.$
	Take a non-negative infinitely differentiable function $\phi_k$ with a compact support $[K,K+1]$.\\
	$\phi_K$ is integrable and one can define $\psi_K (x)=\frac{1}{\int_0^{+\infty}\phi_K(t)\xdif t}\int_x^{+\infty}\phi_K (t)\xdif t$, an infinitely differentiable function. 
	\begin{equation*}
	\begin{cases}
	\psi_K(x)=\frac{1}{\int_0^{+\infty}\phi_K(t)\xdif t}\int_x^{+\infty}\phi_K (t)\xdif t=1 & \text{if } x<K,\\
	\psi_K(x) \in [0,1] &\text {if } x\in [K,K+1],\\
	\psi_K(x)=\frac{1}{\int_0^{+\infty}\phi_K(t)\xdif t}\int_x^{+\infty}\phi_K (t)\xdif t=0 & \text{if }x>K+1.
	\end{cases}
	\end{equation*}
	Here is an illustration of $\psi_K$:
	\begin{figure}[h!]
		\centering
		\includegraphics[width=100mm]{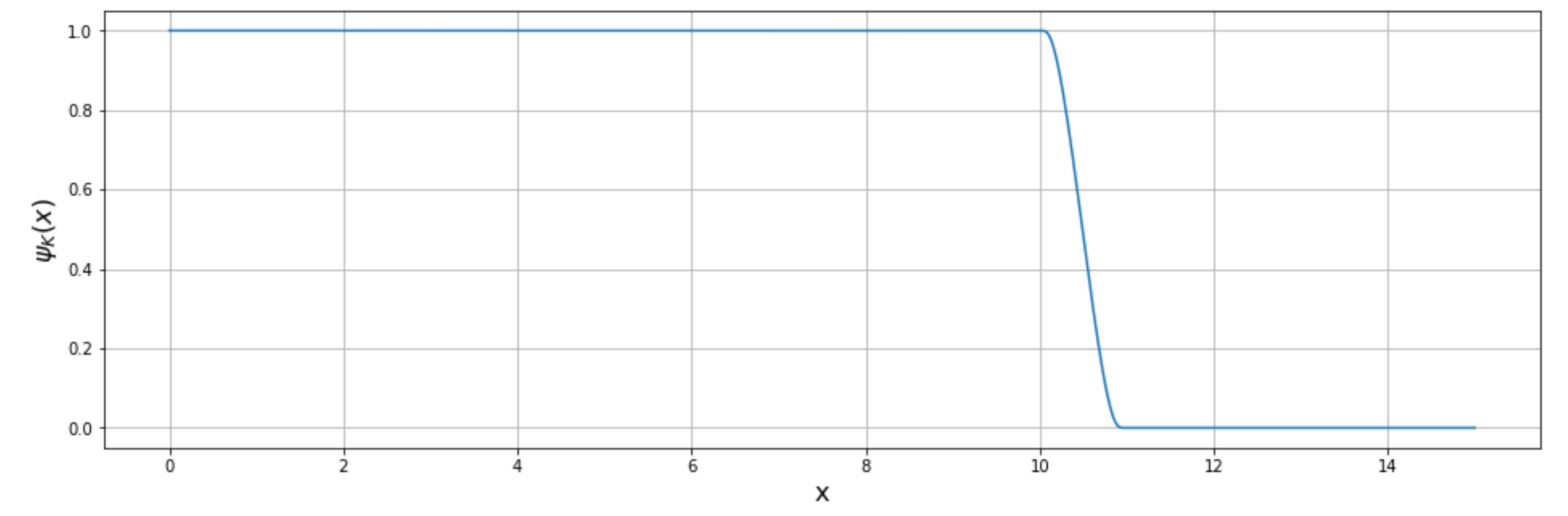}
		\caption{$\psi_K$ is constructed using $\phi_K(x)=\exp(\frac{-1}{1-(2\cdot x-2\cdot K+1)^2})
			$ and $K=10$. }
		\label{fig:my_label}
	\end{figure}
	$\hat C^2_c({\mathbb R_+})$ is clearly a sub-algebra of $\hat C ({\mathbb R_+})$. We prove its density using the locally compact version of the Stone-Weierstrass Theorem.
	\begin{itemize}
		\item Let $c \neq c'$ be two elements of $\mathbb R_+^2$. Assume, without loss of generality that $c < c'$. Set $f(x)=\psi_{cn}(xn)$ where $n$ is such that $\frac{1}{n}< c'-c$.\\
		Clearly $f\in \hat C^2_c({\mathbb R_+})$ and $f(c)=1$ whereas $f(c')=0$. 
		\item For any $c \in \mathbb R_+$, the last function guarantees that $f(c)\neq 0$ thus $\hat C^2_c({\mathbb R_+})$ vanishes nowhere.
	\end{itemize}
	We conclude that $\hat C^2_c({\mathbb R_+})$ is dense in $\hat C ({\mathbb R_+})$ for the norm $\|.\|$.
\end{proof}
\begin{lmm}
	The set of twice differentiable functions with compact support $\hat C^2_c({\mathbb R_+^2})$ is dense in $\hat C ({\mathbb R_+^2})$ for the norm $\|.\|$. 
\end{lmm}
\begin{proof}
	The proof of this lemma is an extension of the previous one. Set
	$$\mathcal B = \{(x,y) \longrightarrow \sum_{k=1}^n f_k (x) g_k(y), n\in \mathbb N^*, (f_k,g_k) \in \hat C _2 ^c (\mathbb R_+)^2\}$$
	a sub-algebra of $\hat C ({\mathbb R_+^2})$ (that is stable by sum, product as well as scalar multiplication). In order to apply the Stone-Weierstrass Theorem one must make sure that $\mathcal B$ separates points and vanishes nowhere. 
	\begin{itemize}
		\item Let $X \neq X'$ be two vectors in $\mathbb R_+^2$. Assume, without loss of generality that their first components $c$ and $c'$ are such that $c < c'$. Set $f(x,y)=\psi_{cn}(xn)$ where $n$ is such that $\frac{1}{n}< c'-c$.\\
		Clearly $f\in \mathcal B$ and $f(X)=1$ where as $f(X')=0$. 
		\item For any $X \in \mathbb R_+^2$, the last function guarantees that $f(X)\neq 0$ thus $\mathcal B$ vanishes nowhere.
	\end{itemize}
	We conclude that $\mathcal B$ and a fortiori $\hat C^2_c({\mathbb R_+^2})$ is dense in $\hat C ({\mathbb R_+^2})$ for the norm $\|.\|$. 
\end{proof}

\subsection{General Results on Continuous Time Markov Processes}
\begin{dfntn}
	A family of bounded linear operators $\big(\mathcal T(t)\big)_{t\geq 0}$ on $\hat C (E)$ is called a semigroup if for each $s,t\geq 0$:
	\begin{itemize}
		\item $\mathcal T(t+s)=\mathcal T(t)\cdot \mathcal T(s),$
		\item $\mathcal T(0)=Id.$
	\end{itemize}
	A semigroup is called:
	\begin{itemize}
		\item A contraction semigroup if $\forall f \in \hat C (E)\text{ and } \forall t\geq 0, \|\mathcal T(t)f\|\leq \|f\|.$
		\item Strongly continuous if $\forall f \in \hat C (E), \lim _{t\rightarrow 0}\|\mathcal T(t)f-f\|=0.$
		\item Conservative if $\mathcal T(t) \mathds 1_E = \mathds 1_E.$ Where $\mathds 1_E$ is the function that takes the value $1$ everywhere.
		\item Positive if $\forall t\geq 0\text{ and } \forall f \in \hat C (E) \text { such that } f\geq 0, \mathcal T(t)f\geq 0.$
		
	\end{itemize}
	If a semigroup has all the previous properties then it is called a Feller semigroup.
\end{dfntn}
\begin{dfntn}
	The infinitesimal generator $\mathcal A$ of a semigroup $\big(\mathcal T(t)\big)_{t\geq 0}$ on $\hat C(E)$ is the linear operator defined by:
	$$\mathcal A f =\lim _{t\rightarrow 0}\frac{\mathcal T(t)f-f}{t},$$
	whenever the limit exists in $\hat C(E)$.\\
	The domain $\mathcal D(\mathcal A)$ is the subset of the functions $f\in \hat C (E)$ for which the limit exists.
\end{dfntn}
\begin{dfntn}\label{core}
	Let $\mathcal A$ be the generator of a Feller semigroup $\mathcal T(t)$ on $\hat C(E)$. Let $D$ be a dense subspace of $\hat C(E)$ with $D \subset \mathcal D (\mathcal A)$. If $\mathcal T(t):D\longrightarrow D $ for all $t\geq 0$, then we say that $D$ is a core for $\mathcal A$.
\end{dfntn}
\begin{rmrk}
	The actual definition of a core is different (cf \cite{MR838085}, page 17), what we have just introduced above is merely a sufficient condition for a subset to be a core. It is sufficient for our application nevertheless.
\end{rmrk}
\subsection{Known Results on the Continuous Time Intensity}
\begin{thrm}
	\begin{enumerate}
		\item Let $(H_t)_{t\in [0,+\infty)}$ be a Hawkes process whose intensity $(\lambda_t)_{t\in [0,+\infty)}$ follows the Markov dynamics of equation \eqref{eq:intensity}. Then $(\lambda_t)_{t\in [0,+\infty)}$ is a Markov process whose semigroup
		$$\mathcal T_e (t) f(x)=\mathbb E[f(\lambda_t)|\lambda_0=x]$$
		is a well defined Feller semigroup that satisfies $\mathcal T_e(t):\hat C({\mathbb R_+})\rightarrow \hat C({\mathbb R_+})$.\\
		The domain of the generator is $\mathcal D(\mathcal A_e)=C^1(\mathbb R_+)$ 
		Moreover, the generator -defined on the set of continuously differentiable functions $C^1({\mathbb R_+})$- is
		$$\mathcal A_e f(\lambda)=\beta (\lambda_\infty - \lambda)f'(\lambda)+\lambda \int \big(f(\lambda+\alpha z)-f(\lambda)\big)d\nu (z).$$
		\item Let $(H_t)_{t\in [0,+\infty)}$ be a Hawkes process whose intensity $(\lambda_t)_{t\in [0,+\infty)}$ follows the Erlang dynamics of equation \eqref{eq:intensityE}. Then $(\lambda_t,\xi_t)_{t\in [0,+\infty)}$ (where $\xi$ is the auxiliary process) is a Markov process whose semigroup
		$$\mathcal T_E (t) f(x,y)=\mathbb E[f(\lambda_t,\xi_t)|\lambda_0=x,\xi_0=y]$$
		is a well defined Feller semigroup that satisfies $\mathcal T_E(t):\hat C({\mathbb R_+^2})\rightarrow \hat C({\mathbb R_+^2})$.\\
		The domain of the generator is $\mathcal D(\mathcal A_E)=C^1(\mathbb R_+^2)$ 
		Moreover, the generator -defined on the set of continuously differentiable functions $C^1({\mathbb R_+^2})$- is
		$$\mathcal A_E f(\lambda,\xi)=\big (\xi+ \beta (\lambda_\infty - \lambda)\big )\partial_\lambda f(\lambda,\xi)-\beta \xi \partial_\xi f(\lambda, \xi)+\lambda \int \big(f(\lambda,\xi+\alpha z)-f(\lambda,\xi)\big)d\nu (z).$$
	\end{enumerate}
\end{thrm}

\begin{proof}
	\begin{enumerate}
		\item If the kernel is an exponential function
		For the proof that $(\lambda_t)_{t\in [0,+\infty)}$ is a Markov process and the expression of its generator we refer to \cite{MR2719785}, section $2.3$.\\
		However, we prove that $\mathcal T(t):\hat C({\mathbb R_+}) \rightarrow \hat C({\mathbb R_+})$ is a Feller semigroup.\\
		Let $t\geq 0$ and $f\in \hat C({\mathbb R_+})$. Start by showing that $\mathcal T(t)f $ is continuous. To do so, let $x\in {\mathbb R_+}$ and a sequence $\epsilon_n\rightarrow 0$ ($\epsilon_n$ must be positive if $x=0$).
		\begin{align*}
		\mathcal T (t)f(x+\epsilon_n)&=\mathbb E[f(\lambda_t)|\lambda_0=x+\epsilon_n],\\
		&=\mathbb E\big[f(\lambda_{\infty}+(x+\epsilon_n-\lambda_{\infty})e^{-\beta t}+\int_{[0,t)}e^{-\beta (t-s)}\xdif L_s)\big].
		\end{align*}
		Since $f$ is continuous and $x+\epsilon_n \rightarrow x$, and given that $f(\lambda_{\infty}+(x+\epsilon_n-\lambda_{\infty})e^{-\beta t}+\int_{[0,t)}e^{-\beta (t-s)}\xdif L_s) \leq \|f\| \in L^1$, one can apply the Dominated Convergence Theorem to conclude that: $$\mathcal T f(x+\epsilon_n)\rightarrow_{n\rightarrow+\infty} \mathcal T f(x).$$
		To prove that $\mathcal T(t)f $ vanishes at infinity we start by setting $\epsilon>0$ and we take $K$ such that $x>K $ implies $|f(x)|\leq \epsilon$. If $x>(K-\lambda_\infty)e^{\beta t}+\lambda_{\infty}$ one has
		$$|f(\lambda_{\infty}+(x+\epsilon_n-\lambda_{\infty})e^{-\beta t}+\int_{[0,t)}e^{-\beta(t-s)}\xdif L_s)|\leq \epsilon.$$
		Thus $|\mathcal T (t)f(x)|\leq \epsilon$ and $\mathcal T(t) : \hat C({\mathbb R_+}) \rightarrow \hat C({\mathbb R_+}).$\\
		Now we prove that the semigroup is Feller.
		\begin{itemize}
			\item Let $f\in \hat C({\mathbb R_+}),$ for all $x\in {\mathbb R_+}$ we have $ |\mathcal T(t)f(x)|\leq \mathbb E[|f(\lambda_t)||\lambda_0=x]\leq \mathbb E [\|f\||\lambda_0=x]\leq \|f\|.$
			Thus $T$ is a contraction.
			\item $\mathcal T(t)\mathds 1_{\mathbb R_+} (x)= \mathbb E [\mathds 1_{\mathbb R_+} (\lambda_{\infty}+(x-\lambda_{\infty})e^{-\beta t}+\int_{[0,t)}e^{-\beta (t-s)}\xdif L_s)|\lambda_0=x]=\mathds 1_{\mathbb R_+}(x)$. Thus $\mathcal T$ is conservative.
			\item If $f\geq 0$ then clearly $\mathbb E[f(\lambda_t)|\lambda_0=x]\geq 0 .$ Thus $\mathcal T$ is positive.
			\item To prove strong continuity, we start by taking $\epsilon>0$, $f\in \hat C ({\mathbb R_+})$ and $f_n\in \hat C^2_c({\mathbb R_+})$ such that $\|f_n-f\|\leq \epsilon$ (cf lemma \ref{density2}). Since we will make $t\rightarrow 0$ it is possible to assume $t<1$.\\
			Using Jensen's inequality:
			\begin{align*}
			\|\mathcal T(t)f_n-f_n\|&=\sup_{x\in {\mathbb R_+}}\big|\mathbb E[f_n(\lambda_{\infty}+(x-\lambda_{\infty})e^{-\beta t}+\int_{[0,t)}e^{-\beta (t-s)}\xdif L_s)-f_n(x)]\big|,\\
			&\leq \sup_{x\in {\mathbb R_+}}\mathbb E[\big|f_n(\lambda_{\infty}+(x-\lambda_{\infty})e^{-\beta t}+\int_{[0,t)}e^{-\beta (t-s)}\xdif L_s)-f_n(x)\big|],\\
			&\leq \mathbb E[\sup_{x\in {\mathbb R_+}} \big|f_n(\lambda_{\infty}+(x-\lambda_{\infty})e^{-\beta t}+\int_{[0,t)}e^{-\beta (t-s)}\xdif L_s)-f_n(x)\big|],\\
			\end{align*}
			Now set $\alpha_x=\lambda_{\infty}+(x-\lambda_{\infty})e^{-\beta t}+\int_{[0,t)}e^{-\beta (t-s)}\xdif L_s$, since we assumed that $t<1$ we have $\lambda_\infty+(x-\lambda_\infty)e^{-\beta} \leq \inf(x,\alpha_x)$.\\
			Using the mean value theorem, there is $\lambda_\infty+(x-\lambda_\infty)e^{-\beta t}\leq \theta_x$ (random) such that 
			\begin{align*}
			f_n(\alpha_x)-f_n(x)&=(\alpha_x-x)\cdot f_n'(\theta_x),\\
			&=\big((x-\lambda_\infty)(e^{-\beta t}-1)+\int_{[0,t)}e^{-\beta (t-s)}\xdif L_s\big)\cdot f_n'(\theta_x).
			\end{align*}
			the function $f_n$ is $C^1$ with compact support. Therefore the following inequalities are obtained:
			\begin{enumerate}
				\item $|f'_n(x)|\leq M$ where $M\in [0,+\infty)$ (deterministic), for any $x\in {\mathbb R_+}$.
				\item $|(x-\lambda_\infty)f'_n(\theta_x)|\leq M'$ where $M'\in [0,+\infty)$ (deterministic), for any $x\in {\mathbb R_+}$. This is due to the fact that $\lambda_\infty+(x-\lambda_\infty)e^{-\frac{1}{\tau}}\leq \theta_x$ which imposes that if $x$ is too large, then $f'(\theta_x)=0$.
			\end{enumerate}
			Moreover, since $e^{-\beta (t-s)}\leq 1$ for $s\in [0,t)$ one has
			\begin{align*}
			\int_{[0,t)}e^{-\beta (t-s)}\xdif L_s &\leq \int_{[0,t)}1\xdif L_s,\\
			&= L_{t-},\\
			&\leq L_t.
			\end{align*}
		\end{itemize}
		
		Combining all these elements yields:
		\begin{align*}
		\|\mathcal T(t)f_n-f_n\|&\leq \mathbb E [\sup_{x\in {\mathbb R_+}}|\big((x-\lambda_\infty)(e^{-\beta t}-1)+\int_{[0,t)}e^{-\beta (t-s)}\xdif L_s\big)\cdot f_n'(\theta_x)|],\\
		&\leq \mathbb E [M'\cdot (1-e^{-\beta t})+M\cdot H_t].
		\end{align*}
		From \cite{MR3084573} we have an explicit expression for $\mathbb E[H_t]$ and we know that $\lim_{t\rightarrow 0} \mathbb E[H_t]=0$, thus the result for $f_n$. Now we extend it by density for the norm $\|\|$:
		\begin{align*}
		\|\mathcal T(t)f-f\|&=\|\mathcal T (t) (f-f_n)+\mathcal T(t)f_n-(f-f_n)-f_n\|,\\
		&\leq \|\mathcal T (t) (f-f_n)-(f-f_n)\|+\|\mathcal T(t)f_n-f_n\|,\\
		&\leq \|\mathcal T (t) (f-f_n)\|+\|(f-f_n)\|+\|\mathcal T(t)f_n-f_n\|.
		\end{align*}
		Finally, since $\mathcal T$ is a contraction, $\|\mathcal T (t) (f-f_n)\|\leq \|(f-f_n)\|$ and we conclude that $$\|\mathcal T(t)f-f\|\rightarrow 0.$$
		\item If the kernel is an Erlang function:\\
		The generator and its domain can be found in \cite{MR4044609}. All the other computations are identical to those of the exponential kernel case.
	\end{enumerate}
\end{proof}
\section{Proof of the Main Result}
The main result (Theorem \ref{MainResult}) is an immediate corollary of the following theorem (Theorem $2.7$ from \cite{MR838085} page 168):
\begin{thrm}\label{convergence}
	Let $E$ be locally compact and separable. For $N=1,2,\cdots$ let $\mu_N(x,\Gamma)$ be a transition function on $E\times \mathcal B (E)$ such that $\mathcal T_N$ defined by
	$$\mathcal T_N f(x)=\int f(y)\mu_N(x,dy),$$
	satisfies $\mathcal T_N:\hat C (E)\longrightarrow \hat C (E)$. Suppose that $\big(\mathcal T(t)\big)_{t\geq 0}$ is a Feller semigroup on $\hat C (E)$. Let $h_N>0$ satisfy $\lim _{N\rightarrow +\infty} h_N=0$ and suppose that for every $f\in \hat C (E)$, $$\lim _{N\rightarrow +\infty}\mathcal T^{\lfloor t/h_N\rfloor}_N f=\mathcal T(t)f, ~~ t\geq 0. $$
	For each $N\geq 1$, let $(Y^N_k)_{k\geq 0}$ be a Markov chain in $E$ with transition function $\mu_N(x,\Gamma)$ and suppose that $Y^N_0$ has a limiting distribution $\nu$. Define $X^N$ by $X^N_t= Y^N_{\lfloor t/h_N \rfloor}$.\\
	Then there is a Markov process $X$ corresponding to $\big(\mathcal T(t)\big)_{t\geq 0}$ with initial distribution $\nu $ and sample paths in $D_E[0,+\infty[$ and $X^N \Longrightarrow X$.
\end{thrm}
In this section we prove that the process $(\tilde \lambda_t^N)_{t\in [0,+\infty)}$ satisfies the conditions of Theorem \ref{convergence}.\\
In the context of this paper, $(Y^N_k)_{k\geq 0}=(l^N_k)_{k\geq 0}$ (or $(l^N_k,a^N_k)_{k\geq 0}$), $X^N=\tilde\lambda^N$ (or $(\tilde \lambda^N, \tilde \xi ^N)$).
\subsection {Initial Condition}
First of all, we fix $l^N_0=x$ (and $a^N_0=0$ if the kernel is an Erlang function) for some $x\in {\mathbb R_+}$ independently from $N$, thus $l^N_0$ does have a limiting distribution $\delta_x$.\\
\subsection{Convergence of the Operators}
Now the trickier part to prove is the convergence of the discrete one-step operator to the Feller semigroup associated with the Hawkes intensity. Unfortunately, we do not know that much about the semigroup nor about the composition of one-step operator with itself. That is why using generators is indispensable.\\
We start this part by mentioning the lemmas (from \cite{MR838085}) that will be used:
\begin{lmm}\label{equivalence}
	Let $L$ be a Banach functional space on $E$.\\
	
	For $N=1,2,\cdots$ let $\mathcal T_N$ be a linear contraction on $L$, let $h_N$ be a positive number and put $\mathcal A_N=h_N^{-1}(\mathcal T_N-Id)$. Assume that $\lim _{N\rightarrow +\infty}h_N=0$. Let $\big((\mathcal T(t))_{t\geq 0}\big)$ be a strongly continuous contraction semigroup on $L$ with generator $\mathcal A$ and let $D$ be a core for $\mathcal A$. Then the following are equivalent:
	\begin{enumerate}
		\item For each $f\in L, \mathcal T_N^{\lfloor t/h_N\rfloor}f\longrightarrow \mathcal T(t)f$ for all $t\geq 0$.
		\item For each $f \in D$ there exists $f_N \in L$ such that $f_N\longrightarrow f$ and $\mathcal A_Nf_N\longrightarrow \mathcal Af$.
	\end{enumerate}
\end{lmm}
\begin{proof}
	Cf \cite{MR838085} page 31.
\end{proof}

\begin{lmm}
	$\hat C^2_c(E)$ is a core for $\mathcal A_{j}$, $j\in \{e,E\}$.
\end{lmm}
\begin{proof}
	Let us start with the exponential kernel case. According to Definition \ref{core} one must show that $\mathcal T (t): \hat C^2_c({\mathbb R +})\longrightarrow \hat C^2_c({\mathbb R +})$, the density has been proven in Lemma \ref{density2}.\\
	Let $t\geq 0 $ and $f\in C^2_c({\mathbb R +})$. There exists $B>0$ such that $x\geq B \Rightarrow f(x)=0$.\\
	If $x\geq (B-\lambda_\infty)e^{\beta t}+\lambda_\infty,$ then $\lambda_t\geq \lambda_\infty+(x-\lambda_\infty)e^{-\beta t}\geq B$.\\ It follows that $\mathcal T(t)f(x)=\mathbb E (f(\lambda_t)|\lambda_0=x)=0$.\\
	The interchangeability of the derivative and the expectation is possible because $\|f'\|<+\infty$ and $\|f''\|<+\infty$, thus the (twice) differentiability of $\mathcal T(t)f$.\\
	If the kernel is an Erlang function, the computations are similar.
\end{proof}
\begin{prpstn}
	\label{onestep}
	$E=\mathbb R_+$ or $\mathbb R_+^2$.\\
	\begin{enumerate}
		\item We assume that the kernel is exponential. The one-step transition operator $\mathcal T^N_e$ associated to $l^N$ and evaluated at a function $f\in \hat C (E)$ is:
		\begin{align*}
		\mathcal T^N_e f(y):=&\mathbb E[f(l_{k+1}^N)|l_k^N=y],\\
		=&f\big(\lambda_\infty(1-e^{-\beta h}\big)+ye^{-\beta h})(1-yh)\mathds 1_{yh<1}\nonumber\\
		&+\int f\big(\lambda_\infty(1-e^{-\beta h})+(y+\alpha z)e^{-\beta h}\big)d\nu(z)(yh\mathds 1_{yh<1}+\mathds 1_{yh\geq1}).
		\end{align*}
		\item If the kernel is an Erlang function, then the one-step transition operator $\mathcal T^N_E$ associated to $(l^N,a^N)$ and evaluated at a function $f\in \hat C (E)$ is:
		\begin{align*}
		\mathcal T^N_E f(y,v)&:=\mathbb E[f(l_{k+1}^N,a_{k+1}^N)|l_k^N=y,a_k^N=v],\\
		=&f\big(\lambda_\infty(1-e^{-\beta h})+ye^{-\beta h}+vh e^{-\beta h},v e^{-\beta h}\big)(1-yh)\mathds 1_{yh<1}\nonumber\\
		&+\int f\big(\lambda_\infty(1-e^{-\beta h})+y e^{-\beta h}+h(v+\alpha z)e^{-\beta h}+,(v+\alpha z)e^{-\beta h}\big)d\nu(z)(yh\mathds 1_{yh<1}+\mathds 1_{yh\geq1}).
		\end{align*}
	\end{enumerate}
\end{prpstn}
\begin{proof}
	Let $N\in \mathbb N^*$ and $T>0$.
	\begin{enumerate}
		\item Set $F(l,u,\zeta)=\lambda_{\infty}(1-e^{-\beta h})+(l+\alpha \zeta \mathds 1_{u<l\cdot h})e^{-\beta h}$. $F$ is clearly measurable and $l^N_{k+1}=F(l^N_k,U^N_{k+1},\zeta^N_{k+1})$, where $l_k^N$ is $\sigma(U^N_i,\xi^N_i, i\in \llbracket 0, k\rrbracket)$ measurable, for any $k\in \mathbb N$. Thus, $(l^N_k)_{k\in \mathbb N}$ is a Markov chain.\\
		When it comes to the one-step transition operator, computing the expected value yields:
		\begin{align*}
		\label{onestep}
		\mathcal T^N _e f(y)&:=\mathbb E[f(l_{k+1}^N)|l_k^N=y],\\
		&=\mathbb E[f\big(\lambda_\infty(1-e^{-\beta h})+(l^N_k+\alpha \xi_{k+1}^N\mathds 1_{U_{k+1}^N<l^N_k\cdot h})e^{-\beta h}\big)|l_k^N=y],\\
		&=\mathbb E\big[f \big(\lambda_\infty(1-e^{-\beta h})+(y+\alpha \xi_{k+1}^N\mathds 1_{U_{k+1}^N<y\cdot h})e^{-\beta h}\big)\big],\\
		\end{align*}
		and since $U^N_{k+1}$ and $\xi^N_{k+1}$ are independent from $l^N_k \in \sigma(U^N_i,\xi^N_i, i\in \llbracket 0, k\rrbracket)$ and since $\mathds 1_{U_{k+1}^N<y\cdot h}$ is a Bernoulli variable with parameter $y\cdot h$ independent from $\xi_{k+1}^N$:
		\begin{align*}
		\mathcal T^N_e f(y)=&f\big(\lambda_\infty(1-e^{-\beta h}\big)+ye^{-\beta h})(1-yh)\mathds 1_{yh<1}\nonumber\\
		&+\int f\big(\lambda_\infty(1-e^{-\beta h})+(y+\alpha z)e^{-\beta h}\big)d\nu(z)(yh\mathds 1_{yh<1}+\mathds 1_{yh\geq1}).
		\end{align*}
		\item Set 
		\begin{equation*}
		F(l,a,u,\zeta)=   
		\begin{pmatrix}
		\lambda_\infty(1-e^{-\beta h})+l e^{-\beta h} +(a+\alpha \zeta \mathds 1_{u<l\cdot h}) h e^{-\beta h}\\
		(a+\alpha \zeta \mathds 1_{u<l\cdot h}) e^{-\beta h}
		\end{pmatrix}
		\end{equation*}
		a measurable function. Clearly $(l^N_{k+1}, a^N_{k+1})=F(l^N_k, a^N_k,U^N_{k+1},\zeta ^N_{k+1})$ where $(l^N_k,a^N_k)$ is $\sigma(U^N_i,\xi^N_i, i\in \llbracket 0, k\rrbracket)$ measurable, for any $k\in \mathbb N$. Thus, $(l^N_k,a^N_k)_{k\in \mathbb N}$ is a Markov chain.\\
		The Erlang one-step generator can be obtained just like exponential one. 
	\end{enumerate}
\end{proof}

\begin{thrm}
	$E=\mathbb R_+$ or $\mathbb R_+^2$.\\
	Let $f\in \hat C^2_c(E)$, $\mathcal A$ be the generator of a Hawkes intensity and $\mathcal T^N_j$ where $j\in \{e,E\}$ the operator described in Proposition \ref{onestep}. Then
	$$\|\frac{\mathcal T^N_j f-f}{h_N}-\mathcal A_j f\|\longrightarrow 0.$$
\end{thrm}
\begin{proof}
	First we remind that $h=h_N$.\\
	\textbf {If the kernel is an exponential function}\\
	$E=\mathbb R_+$.\\
	Let $f\in \hat C^2_c(E)$ be a fixed function. We start by giving an alternative expression for $\frac{\mathcal T^N_e f(y)-f(y)}{h_N}$ for a fixed $y\in {E}$:
	\begin{align*}
	\frac{ \mathcal T^N_e f(y)-f(y)}{h_N}
	=&\big[f(\lambda_\infty(1-e^{-\beta h})+ye^{-\beta h})(1-yh)\mathds 1_{yh<1}\nonumber\\
	&+\int f(\lambda_\infty(1-e^{-\beta h})+(y+\alpha z)e^{-\beta h})d\nu(z)(yh\mathds 1_{yh<1}+\mathds 1_{yh\geq1})-f(y)\big]h_N^{-1},\\
	=&\big[f(\lambda_\infty(1-e^{-\beta h})+ye^{-\beta h})(1-yh)+\int f(\lambda_\infty(1-e^{-\beta h})+(y+\alpha z)e^{-\beta h})d\nu(z)yh -f(y)\big]h_N^{-1}\mathds 1_{yh<1}\nonumber\\
	&+\big[\int f(\lambda_\infty(1-e^{-\beta h})+(y+\alpha z)e^{-\beta h})-f(y) d\nu(z)\big]h_N^{-1}\mathds 1_{yh\geq1},\\
	\end{align*}
	since $\int d\nu(z)=1$.\\
	Now using the fact that $f$ is twice differentiable, we use Taylor expansion with a Lagrange remainder:
	\begin{align*}
	f\big(\lambda_\infty(1-e^{-\beta h})+ye^{-\beta h}\big)=&f(y)+f'(y)\big[(\lambda_\infty-y)(1-e^{-\beta h})\big]\\&+ \frac{1}{2}f''(\theta_y)\big[(\lambda_\infty-y)(1-e^{-\beta h})\big]^2,\\
	\end{align*}
	where $\theta_y \in \big[\inf (y,\lambda_\infty(1-e^{-\beta h})+ye^{-\beta h}),\sup (y,\lambda_\infty(1-e^{-\beta h})+ye^{-\beta h})\big]\subset [ye^{-\beta T},\lambda_\infty(1-e^{-\beta T})+y]$.\\
	The last inclusion will be used later.\\
	We apply another Taylor expansion to obtain:
	$$f\big(\lambda_\infty(1-e^{-\beta h})+(y+\alpha z)e^{-\beta h}\big)=f(y+\alpha z)+f'(\gamma_y)\big[(\lambda_\infty-y-\alpha z)(1-e^{-\beta h})\big]$$
	where $\gamma_y \in [(y+\alpha z)e^{-\beta T},\lambda_\infty(1-e^{-\beta T})+y+\alpha z].$\\
	Thus we have:
	\begin{align*}
	\frac{ \mathcal T^N_e f(y)-f(y)}{h_N}\mathds 1_{yh<1}=&\big (\cancel {f(y)}+f'(y)\big[(\lambda_\infty-y)(1-e^{-\beta h})\big]+ \frac{1}{2}f''(\theta_y)\big[(\lambda_\infty-y)(1-e^{-\beta h})\big]^2\\
	&-yhf(y)-yhf'(y)\big[(\lambda_\infty-y)(1-e^{-\beta h})\big]-yh \frac{1}{2}f''(\theta_y)\big[(\lambda_\infty-y)(1-e^{-\beta h})\big]^2\\
	&+yh\int f(y+\alpha z)+f'(\gamma_y)\big[(\lambda_\infty-y-\alpha z)(1-e^{-\beta h})\big]d\nu(z)-\cancel {f(y)}\big)h^{-1}\mathds 1_{yh<1},
	\end{align*}
	a Taylor expansion for small $h$ yields $1-e^{-\beta h}=\beta h+O(h^2)=O(h)$, thus:
	\begin{align*}
	\frac{ \mathcal T^N_e f(y)-f(y)}{h_N}\mathds 1_{yh<1}  =&\big(f'(y)(\lambda_\infty-y)\big(\beta+O(h)\big) +\frac{1}{2}f''(\theta_y)(\lambda_\infty-y)^2\frac{1-e^{-\beta h}}{h}O(h),\\
	&-yf(y)-yf'(y)(\lambda_\infty-y)O(h)-y\frac{1}{2}f''(\theta_y)(\lambda_\infty-y)(1-e^{-\beta h})O(h)\\
	&+y\int f(y+\alpha z)d\nu(z)+y\int f'(\gamma_y)(\lambda_\infty-y-\alpha z) d\nu (z)O(h)\big )\mathds 1_{yh<1},\\
	=&\big (\beta(\lambda_\infty-y)f'(y)+y\int \big(f(y+\alpha z)-f(y)\big)d\nu (z) \big)\mathds 1_{yh<1}+R(y)O(h)\mathds 1_{yh<1},\\
	=&\mathcal A_e f(y)\mathds 1_{yh<1}+R(y)O(h)\mathds 1_{yh<1},
	\end{align*}
	The remainder $R$ has the expression:
	\begin{align*}
	R(y)=&f'(y)(\lambda_\infty-y)+\frac{1}{2}f''(\theta_y)(\lambda_\infty-y)^2\frac{1-e^{-\beta h}}{h}\\
	&+yf'(y)(\lambda_\infty-y)-y\frac{1}{2}f''(\theta_y)(\lambda_\infty-y)^2(1-e^{-\beta h})
	+y\int f'(\gamma_y)(\lambda_\infty-y-\alpha z) d\nu (z).
	\end{align*}
	Note that the remainder is bounded. Remember that $f$ is continuous with compact support (so are its derivatives), so the function $y\rightarrow \sup_{x \in [ye^{-\beta T},\lambda_\infty(1-e^{-\beta T})+y]} |f''(x)|$ is also continuous with compact support, which in turn means that $y\rightarrow y^k\sup_{x \in [y e^{-\beta T},\lambda_\infty(1-e^{-\beta T})+y]} |f''(x)|$ for $k=2,3$ is also compact support, thus bounded by a positive constant independently from $y$. Applying the same logic to the terms involving $f'$ yields:
	$$\big|R(y)\big|\leq C,$$
	where $C$ is a positive constant independent from $y$.
	After all these computations, we prove the uniform convergence for.\\
	Let $\epsilon >0$ and $B$ a constant such that $y>B $ implies $f(y)=f'(y)=0$.\\
	For all $y\in {\mathbb R_+}$ we have:
	\begin{align*}
	|\frac{\mathcal T^N_e f(y)-f(y)}{h_N}-\mathcal A_e f(y)|&=\big|\big(\frac{\mathcal T^N_e f(y)-f(y)}{h_N}-\mathcal A_e f(y)\big)\mathds 1_{yh<1}+\big(\frac{\mathcal T^N_e f(y)-f(y)}{y}-\mathcal A_e f(y)\big)\mathds 1_{yh\geq1}\big|\\
	&\leq \big|\big(\frac{\mathcal T^N_e f(y)-f(y)}{h_N}-\mathcal A_e f(y)\big)\big|\mathds 1_{yh<1}+\big|\big(\frac{\mathcal T^N_e f(y)-f(y)}{h_N}-\mathcal A_e f(y)\big)\big|\mathds 1_{yh\geq1}.
	\end{align*}
	Now we plug in the inequality obtained previously for the term in front of $\mathds 1_{yh<1}$ and expand the one in front of $\mathds 1_{yh\geq 1}$:
	\begin{align*}
	|\frac{\mathcal T^N_e f(y)-f(y)}{h_N}-\mathcal A_e f(y)|&\leq Ch \mathds 1_{yh<1}+ \big|\big[\int f\big(\lambda_\infty(1-e^{-\beta h})+(y+\alpha z)e^{-\beta h}\big)-f(y)d\nu (z)\big]h_N^{-1}-\mathcal A_e f(y)\big| \mathds 1_{yh\geq 1}
	\end{align*}
	where $O(h)$ has been absorbed by the constant $C$.\\
	
	Let $N_0$ be the integer such that if $N\geq N_0$, then $\frac{N}{T}e^{-\beta \frac{T}{ N}}\geq B$. Such integer exists because $\frac{N}{T}e^{-\beta\frac{T}{ N}}\geq \frac{N}{T}e^{-\beta T}\rightarrow +\infty$ as $N\rightarrow +\infty$.\\
	
	Set $N_1=\lfloor \frac{CT}{\epsilon}\rfloor+TB+N_0$.\\
	For every $N\geq N_1$ and every $y\in {\mathbb R_+}$ only one of these two scenarios is possible:
	\begin{itemize}
		\item $\mathds 1_{yh<1}=1$ and $\mathds 1_{yh\geq 1}=0$, so $  \big|\frac{\mathcal T^N_e f(y)-f(y)}{h_N}-\mathcal A_e f(y)\big| \leq \epsilon$ because $N\geq N_1\geq  \lfloor \frac{CT}{\epsilon}\rfloor$.
		\item $\mathds 1_{yh<1}=0$ and $\mathds 1_{yh\geq 1}=1$ which means $y\geq \frac{N}{T}$, thus $y\geq B$ and $(y+\alpha z)e^{-\beta h}\geq \frac{N}{T}e^{-\beta \frac{T}{ N}}\geq B$.\\
		Therefore $f(\lambda_\infty(1-e^{-\beta h})+(y+\alpha z)e^{-\beta h})=f(y)=f'(y)=0$ $\forall z \in \mathbb R _+$.\\
		Which leads to $\big|\frac{\mathcal T^N_e f(y)-f(y)}{h_N}-\mathcal A_e f(y)\big|\leq 0$.
	\end{itemize}
	Each scenario leads to the same result: $\big|\frac{\mathcal T^N_e f(y)-f(y)}{h_N}-\mathcal A_e f(y)\big|\leq \epsilon.$\\
	In conclusion, since the rank $N_1$ is independent from the choice of $y$, one can deduce that $\forall N\geq N_1$
	$$\|\frac{\mathcal T^N_e f-f}{h_N}-\mathcal A_e f\|\leq \epsilon.$$
	\textbf{If the kernel is an Erlang function}\\
	$E=\mathbb R_+^2.$ 
	Let $f\in \hat C^2_c(E)$ be a fixed function. We start by giving an alternative expression for $\frac{\mathcal T^N_E f(y)-f(y)}{h_N}$ for a fixed $(y,v)\in {E}$:
	\begin{align*}
	\frac{\mathcal T^N_E f(y,v)-f(y,v)}{h_N}=& \big [ f\big(\lambda_\infty(1-e^{-\beta h})+ye^{-\beta h}+vh e^{-\beta h},v e^{-\beta h}\big)(1-yh)\mathds 1_{yh<1}\nonumber\\
	&+\int f\big(\lambda_\infty(1-e^{-\beta h})+y e^{-\beta h}+h(v+\alpha z)e^{-\beta h}+,(v+\alpha z)e^{-\beta h}\big)d\nu(z)\\ & (yh\mathds 1_{yh<1}+\mathds 1_{yh\geq1})-f(y,v)\big] h_N^{-1},\\
	=& \big [f\big(\lambda_\infty(1-e^{-\beta h})+ye^{-\beta h}+vh e^{-\beta h},v e^{-\beta h}\big)(1-yh) \\
	&+\int f\big(\lambda_\infty(1-e^{-\beta h})+y e^{-\beta h}+h(v+\alpha z)e^{-\beta h}+,(v+\alpha z)e^{-\beta h}\big)d\nu(z) yh-f(y,v)\big]h_N^{-1} \mathds 1_{yh<1} \\
	&+\big [\int f\big(\lambda_\infty(1-e^{-\beta h})+y e^{-\beta h}+h(v+\alpha z)e^{-\beta h}+,(v+\alpha z)e^{-\beta h}\big)-f(y,v)d\nu(z) \big ]h_N^{-1} \mathds 1_{yh\geq1}
	\end{align*}
	since $\int d\nu(z)=1$.\\
	Now using the fact that $f$ is twice differentiable, we use Taylor expansion with a Lagrange remainder:
	\begin{align*}
	f\big(\lambda_\infty(1-e^{-\beta h})+ye^{-\beta h}+vh e^{-\beta h},v e^{-\beta h}\big)=& f(y,v)+\big ((\lambda_{\infty}-y)(1-e^{-\beta h})+vh e^{-\beta h}\big)\partial_\lambda f(y,v)\\
	&+v(e^{-\beta h}-1)\partial _{\xi} f(y,v)+\frac{1}{2}v^2(e^{-\beta h}-1)^2\partial _{\xi \xi} f(\theta_y,\theta_v)\\
	&+\frac{1}{2}\big ((\lambda_{\infty}-y)(1-e^{-\beta h})+vh e^{-\beta h}\big)^2\partial_{\lambda \lambda} f(\theta_y,\theta_v)\\
	&+\big ((\lambda_{\infty}-y)(1-e^{-\beta h})+vh e^{-\beta h}\big)v(e^{-\beta h}-1) \partial_{\lambda\xi} f(\theta_y,\theta_v),
	\end{align*}
	where $\theta_y \in [y e^{-\beta T},\lambda_\infty(1-e^{-\beta h})+y+vh]$ and $\theta_v \in [v e^{-\beta T},v]$.\\
	Just like the exponential case, it is possible to bound all the second order terms by a constant, thus:
	\begin{align*}
	f\big(\lambda_\infty(1-e^{-\beta h})+ye^{-\beta h}+vh e^{-\beta h},v e^{-\beta h}\big)=& f(y,v)+\big ((\lambda_{\infty}-y)(1-e^{-\beta h})+vh e^{-\beta h}\big)\partial_\lambda f(y,v)\\
	&+v(e^{-\beta h}-1)\partial _{\xi} f(y,v)+R_1(y,v) h^2,\\
	\end{align*}
	where $R_1$ is a compact support function that contains all the second derivatives.
	We apply another Taylor expansion to obtain:
	\begin{align*}
	f\big(\lambda_\infty(1-e^{-\beta h})+y e^{-\beta h}+h(v+\alpha z)e^{-\beta h}+,(v+\alpha z)e^{-\beta h}\big)=& f(y,v+\alpha z)+ (v+\alpha z)(e^{-\beta h}-1) \partial _\xi f(\gamma_y,\gamma_v)\\
	&+\big((\lambda_\infty -y) (1-e^{-\beta h})+h(v+\alpha z)e^{-\beta h} \big)\partial _\xi f(\gamma_y,\gamma_v),
	\end{align*}
	where $\gamma_y \in [ye^{-\beta T}, \lambda_\infty(1-e^{-\beta h})+y+(v+\alpha z)h]$ and $\gamma_v \in \big[(v+\alpha z) e^{-\beta T},(v+\alpha z)\big]$. And it is possible to write it under the form:
	\begin{align*}
	f\big(\lambda_\infty(1-e^{-\beta h})+y e^{-\beta h}+h(v+\alpha z)e^{-\beta h},(v+\alpha z)e^{-\beta h}\big)=& f(y,v+\alpha z)+R_2(y,v,z) O(h),
	\end{align*}
	where $R_2$ is not necessarily bounded as $z$ goes to infinity but it is not a problem since $\int z d\nu (z)$ is bounded.\\
	Hence:
	\begin{align*}
	\frac{\mathcal T^N_E f(y,v)-f(y,v)}{h_N}\mathds 1_{yh<1}=&\big [ \big (\cancel{f(y,v)} +\big ((\lambda_{\infty}-y)(1-e^{-\beta h})+vh e^{-\beta h}\big)\partial_\lambda f(y,v)\\
	&+v(e^{-\beta h}-1)\partial _{\xi} f(y,v)+R_1(y,v) O(h^2) -\cancel{f(y,u)}\big ) (1-yh)\big ] h^{-1} \mathds 1_{yh<1}\\
	&+\big [yh \int f(y,v+\alpha z) +R_2 (y,v,z) O(h) -f(y,v) d\nu(z)\big]h^{-1} \mathds 1 _{yh<1},\\
	=&\big [ \big ( \big ((\lambda_{\infty}-y)\big(\beta h +O(h^2)\big)+vO(h) \big)\partial_\lambda f(y,v)\\
	&+v\big(-\beta h +O(h^2)\big)\partial _{\xi} f(y,v)+R_1(y,v) O(h^2) \big ) (1-yh)\big ] h^{-1} \mathds 1_{yh<1}\\
	&+\big [y\cancel{h} \int f(y,v+\alpha z) -f(y,v) d\nu(z) +\cancel{h}\int yR_2 (y,v,z) d\nu(z)O(h)\big]\cancel{h^{-1}} \mathds 1 _{yh<1},\\
	=&\mathcal A _E f(y,v) \mathds 1_{yh<1} +R(y,v) O(h)\mathds 1_{yh<1}.
	\end{align*}
	The remainder $R(y,v)$ contains $R_1(y,v)$ and $\int R_2(y,v,z) d\nu(z)$ as well as their products with $y$ and $v$. It is a compact support functions thus it is bounded by a constant $C$ independent from $y, v$ and $h$.\\
	Let $\epsilon >0$ and $B$ a constant such that $y>B $ or $v>B$ implies $f(y,v)=f'(y,v)=0$ .\\
	For all $y\in {\mathbb R_+}$ we have:
	\begin{align*}
	|\frac{\mathcal T^N_E f(y)-f(y)}{h_N}-\mathcal A_E f(y)|&=\big|\big(\frac{\mathcal T^N_E f(y)-f(y)}{h_N}-\mathcal A_E f(y)\big)\mathds 1_{yh<1}+\big(\frac{\mathcal T^N_E f(y)-f(y)}{y}-\mathcal A_E f(y)\big)\mathds 1_{yh\geq1}\big|\\
	&\leq \big|\big(\frac{\mathcal T^N_E f(y)-f(y)}{h_N}-\mathcal A_E f(y)\big)\big|\mathds 1_{yh<1}+\big|\big(\frac{\mathcal T^N_E f(y)-f(y)}{h_N}-\mathcal A_E f(y)\big)\big|\mathds 1_{yh\geq1}.
	\end{align*}
	Now we plug in the inequality obtained previously for the term in front of $\mathds 1_{yh<1}$ and expand the one in front of $\mathds 1_{yh\geq 1}$:
	\begin{align*}
	|\frac{\mathcal T^N_E f(y,v)-f(y,v)}{h_N}-\mathcal A_E f(y,v)|\leq& Ch \mathds 1_{yh<1} \\&+ \big|\big[\int f\big(\lambda_\infty(1-e^{-\beta h})+ye^{-\beta h}+h(v+\alpha z)e^{-\beta h},(v+\alpha z)e^{-\beta h}\big)-f(y,v)d\nu (z)\big]h_N^{-1}\\&-\mathcal A_e f(y,v)\big| \mathds 1_{yh\geq 1}
	\end{align*}
	where $O(h)$ has been absorbed by the constant $C$.\\
	Let $N_0$ be the integer such that if $N\geq N_0$, then $\frac{N}{T}e^{-\beta \frac{T}{ N}}\geq B$. Such integer exists because $\frac{N}{T}e^{-\beta\frac{T}{ N}}\geq \frac{N}{T}e^{-\beta T}\rightarrow +\infty$ as $N\rightarrow +\infty$.\\
	
	Set $N_1=\lfloor \frac{CT}{\epsilon}\rfloor+TB+N_0$.\\
	For every $N\geq N_1$ and every $(y,v)\in E$ only one of these two scenarios is possible:
	\begin{itemize}
		\item $\mathds 1_{yh<1}=1$ and $\mathds 1_{yh\geq 1}=0$, so $  \big|\frac{\mathcal T^N_e f(y,v)-f(y,v)}{h_N}-\mathcal A_e f(y,v)\big| \leq \epsilon$ because $N\geq N_1\geq  \lfloor \frac{CT}{\epsilon}\rfloor$.
		\item $\mathds 1_{yh<1}=0$ and $\mathds 1_{yh\geq 1}=1$ which means $y\geq \frac{N}{T}$, thus $y\geq B$ and $ye^{-\beta h}\geq \frac{N}{T}e^{-\beta \frac{T}{ N}}\geq B$.\\
		Therefore $f(\lambda_\infty(1-e^{-\beta h})+ye^{-\beta h}+(v+\alpha z)e^{-\beta h},(v+\alpha z)e^{-\beta h})=f(y,v)=\partial _{\lambda}f(y,v)=\partial_{\xi}f(y,v)=0$ $\forall z \in \mathbb R _+$.\\
		Which leads to $\big|\frac{\mathcal T^N_E f(y,v)-f(y,v)}{h_N}-\mathcal A_E f(y,v)\big|\leq 0$.
	\end{itemize}
	In conclusion 
	$$\|\frac{\mathcal T^N_E f-f}{h_N}-\mathcal A_E f\|\leq \epsilon.$$
\end{proof}
\section{Conclusion}
We have proven that the DTHP converges weakly to a time continuous Hawkes process in the case the kernel is an exponential or an Erlang function. The following figure shows a trajectory of a DTHP with a small time step.

\begin{figure}[h!]
	\centering
	\includegraphics[width=130mm]{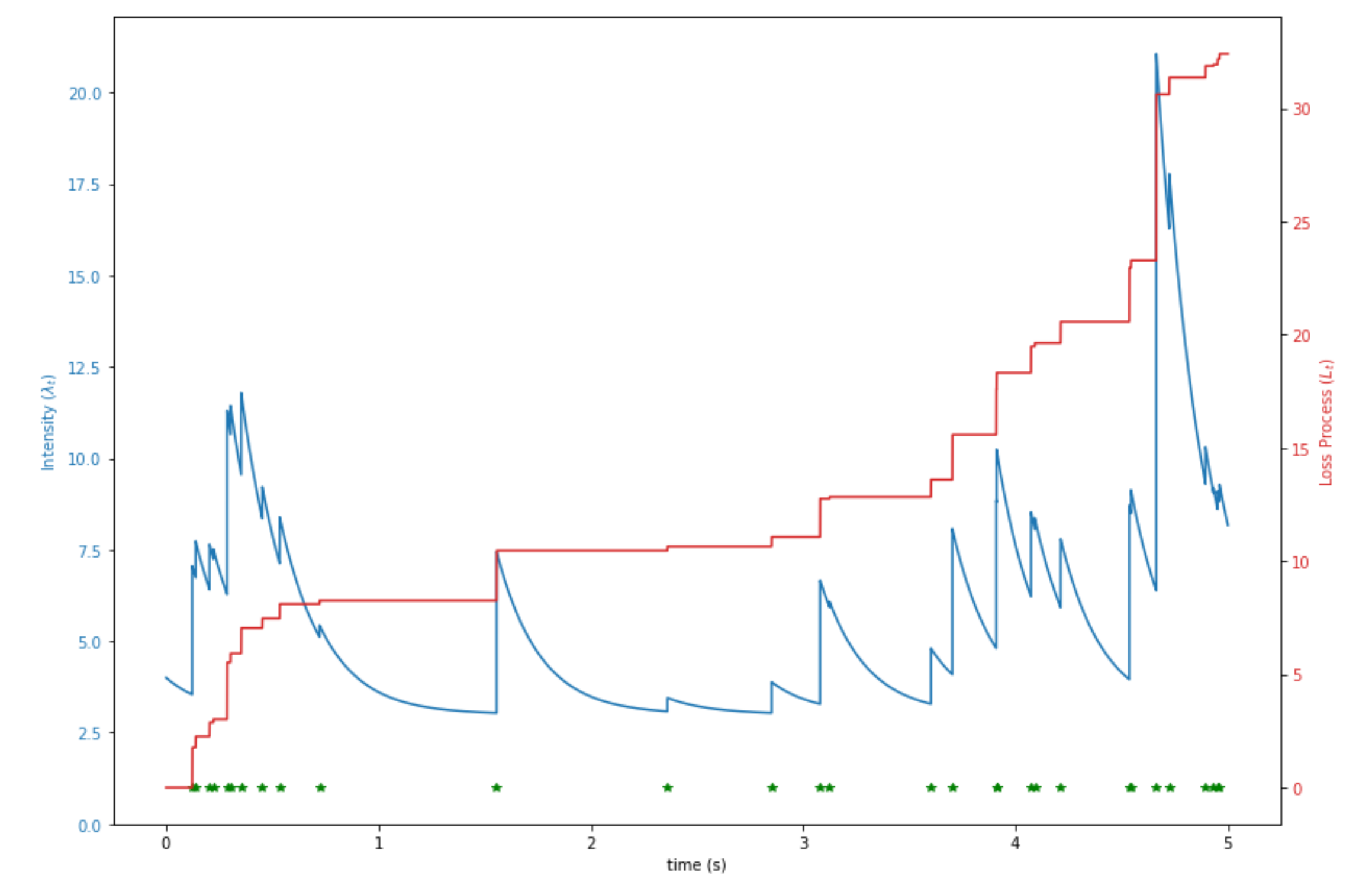}
	\caption{A trajectory of the loss process as well as the intensity in the case of an exponential kernel for $\alpha=2$, $\beta=5$, $\lambda_{\infty}=3$, $\lambda_0=4$ and $N=100000$ points. The financial losses follow an exponential distribution of rate one. The green stars show the jumping times. Note how they are identical for $L_t$ and $\lambda_t$ and exhibit a clustering behaviour.}
	\label{fig:my_label}
\end{figure}

This result is generalisable to a wider class of Hawkes processes like the multivariate Hawkes process whose kernels are exponential/Erlang functions or in the case of a higher order Erlang kernel $\phi(u)=\alpha u^n e^{-\beta u}$ with $n \geq 2$. However, despite being of the same nature, computations for these classes are way too heavy and repetitive to be included in this document.\\

It is also worth mentioning that this convergence does not have a quantified speed yet. 
It would be interesting to have an upper bound on the distance between the two processes as a function of the time step.


\end{document}